\documentclass{amsart}
\usepackage[latin1]{inputenc}
\usepackage{graphicx}
\usepackage{amssymb}

\newtheorem{thm}{Theorem}
\newtheorem{cor}[thm]{Corollary}
\newtheorem{lem}[thm]{Lemma}
\newtheorem{prop}[thm]{Proposition}

\theoremstyle{definition}
\newtheorem{defn}[thm]{Definition}
\theoremstyle{remark}

\numberwithin{equation}{section}

\newcommand{\set}[1]{\left\{#1\right\}}

\newcommand{\To}{\longrightarrow}

\begin{document}
\setcounter{tocdepth}{1}


\title[]{Lexicographic products as compact spaces of the first Baire class}
\author{Antonio Avil\'es}
\address{Universidad de Murcia, Departamento de Matem\'{a}ticas, Campus de Espinardo 30100 Murcia, Spain.} \email{avileslo@um.es}

\author{Stevo Todorcevic}
\address{Institut de Mathématiques de Jussieu, UMR 7586, 2 place Jussieu - Case 247, 75222 Paris Cedex 05, France. Department of Mathematics, University of Toronto, Toronto, Canada, M5S 3G3.}%
\email{stevo@math.toronto.edu}
\thanks{First author supported by projects MTM2017-86182-P (AEI, Spain and ERDF/FEDER, EU) and 20797/PI/18 (Fundaci\'{o}n S\'{e}neca, ACyT Regi\'{o}n de Murcia). Second Author is partially supported by Grants from NSERC (455916) and CNRS (IMJ-PRG UMR7586).}



\begin{abstract}
We use lexicographic products to give examples of compact spaces of first Baire class functions on a compact metric space that cannot be represented as spaces of functions with countably many discontinuities.
\end{abstract}

\maketitle

\section{Introduction}

A compact Hausdorff space can be always represented as a compact set of functions $f:X\To\mathbb{R}$ on a certain set $X$, in the pointwise topology. 

\begin{defn}
	A compact space $K$ is Rosenthal if it can be represented as a compact set of functions $f:X\To\mathbb{R}$ of the first Baire class on a Polish space $X$.
\end{defn}

We consider the following two subclasses of the class $\mathcal{R}$ of Rosenthal compacta:

\begin{itemize}
	\item $K$ belongs to the class $\mathcal{RK}$ if $K$ can be represented as a compact set of functions $f:X\To\mathbb{R}$ of the first Baire class on a compact metric space $X$.
	
	\item $K$ belongs to the class $\mathcal{CD}$ if $K$ can be represented as a compact set of functions $f:X\To\mathbb{R}$ with countably many discontinuities on a compact metric space $X$.
\end{itemize} 

 We have that $\mathcal{CD}\subset \mathcal{RK} \subset \mathcal{R}$. Let us mention that any pointwise compact set of functions with countably many discontinuities on a Borel subset $X$ of a Polish space belongs to the class $\mathcal{CD}$ \cite[Proposition 1]{MarPolRACSAM}. The class $\mathcal{CD}$ was considered by Haydon, Molt\'{o} and Orihuela \cite{HMO}, who proved that the space of continuous functions $C(K)$ admits a locally uniformly rotund renorming whenever $K$ is separable and $K\in \mathcal{CD}$. There is a nonseparable Rosenthal compactum for which $C(K)$ fails to have a LUR renorming \cite{TodLUR}, but it is unknown if such a renorming exists for all separable Rosenthal compacta. For the case considered in this note, if $K$ is a countable lexigraphic product of compact real intervals then $C(K)$ admits a LUR renorming. This is a result of Haydon, Jayne, Namioka and Rogers \cite[Theorem 3.1]{HJNR}.
 
 Although the main result was stated for the class $\mathcal{CD}$, it was left unclear in \cite{HMO} whether this class could coincide with the whole class of Rosenthal compacta. This is not the case, as Pol \cite{PolPAMS86} gave an example of compact space in $\mathcal{R}\setminus \mathcal{RK}$. This was the separable scattered compact space of height three associated with an analytic non Borel almost disjoint family of subsets of $\omega$. Another example given by Marciszewski and Pol \cite{MarPolRACSAM}
is a continuous image of the split interval, obtained by glueing back together all twin couples of a coanalytic non Borel set of points of the unit interval. We will show in this note that the classes $\mathcal{CD}$ and $\mathcal{RK}$ are also different. The lexicographic product of three closed intervals, endowed with the order topology, serves as an example. This space is nonseparable, but we will see that this kind of lexicographic products can be supplemented to create a separable Rosenthal compactum in the class $\mathcal{RK}_0$:
 
 \begin{defn}\label{defRK0}
 	$K$ belongs to the class $\mathcal{RK}_0$ if $K$ can be represented as a compact set of functions $f:X\To\mathbb{R}$ of the first Baire class on a compact metric space $X$, which is the closure of a (countable) set of continous functions on $X$.
 \end{defn}

We wrote the word \emph{countable} in brackets because if there is a dense set of continuous functions, then a countable one also exists. This is because the space of continuous functions $C_p(X)$ is hereditarily separable, since it has a countable network, cf. \cite[Theorem I.1.3]{Arkhan}. This subclass of Rosenthal compacta is the most closely connected to Banach space theory. A compact space $K$ belongs to $\mathcal{RK}_0$ if and only if it is homeomorphic to a weak$^\ast$ compact subset of $E^{\ast\ast}$ for some separable Banach space $E$ that does not contain $\ell_1$, cf. \cite{Marindex}.

Although they do not belong to the class $\mathcal{CD}$, we will see that small lexicographical products are subspaces of continuous images of compact separable spaces in $\mathcal{CD}$. We do not know if every separable Rosenthal compactum is a continuous image of a separable compact space in $\mathcal{CD}$.

Our key observation is that a compact space in the class $\mathcal{CD}$ can be mapped onto a metric space in such a way that preimage of every point is a Corson compactum. This is the topological propery that will be killed in order to check that the examples provided are not in $\mathcal{CD}$.

We may mention here that there has been a study, completed in \cite{Elekes}, of which linearly ordered spaces can be represented as functions of the first Baire class with the pointwise order. This gives a class of linear orders that is substantially larger than countable lexicographic products of intervals. It should be noted, however, that the property that we are studying here is rather different because we are looking at the topology rather than the order, and compactness is essential.

\section{Mappings onto metric spaces with small fibers}

Given a class $\mathfrak{C}$ of compact spaces, we say that a compact space $K$ is a $\mathfrak{C}$-to-one preimage of a metric space if there exists a continous function $f:K\To M$ onto a metric space $M$ such that $f^{-1}(x)\in \mathfrak{C}$ for every $x\in M$. The two-to-one preimages of metric spaces (when $\mathfrak{C}$ are the spaces of at most two points) are involved in several structural results for separable Rosenthal compacta \cite{1BC}, and some generalizations exist for finite-to-one preimages \cite{APT}. We will consider now the class of Corson compacta. A compact space $K$ is Corson if can be represented as compact subset of a power $K\subset \mathbb{R}^I$ in such a way that for every $x,y\in K$, the set $\{i\in I : x_i \neq y_i\}$ is at most countable.

\begin{prop}
	If $K$ is a compact set of functions with countably many discontinuities on a Polish space, then $K$ is a Corson-to-one preimage of a metric space.
\end{prop}

\begin{proof}
	Let $K$ be a pointwise compact set of functions $f:X\To \mathbb{R}$, with $X$ a Polish space. Fix a countable dense set $D\subset X$. The restriction map $r:K \To \mathbb{R}^D$ gives a continuous map into a metric space. If $r(f) = r(g)$, then $f$ and $g$ coincide on the points of common continuity of $f$ and $g$, since they coincide on $D$. Thus, all functions in a given fiber $F = r^{-1}(r(f))$ coincide with $f$ in all but countably many points. This implies that each fiber is a Corson compactum.
\end{proof}

\section{lexicographic order and first Baire class functions}

Let us fix a countable ordinal $\alpha<\omega_1$. We consider the countable power $\{0,1\}^\alpha$, which is a compact metrizable space when endowed with the product topology. We also consider the lexicographic order $\prec$ on it, $x\prec y$ when $x_\beta<y_\beta$ for the first ordinal $\beta$ where $x_\beta\neq y_\beta$. Define

$$\mathcal{H}_\alpha = \left\{f:\{0,1\}^\alpha\To [0,1] :  x\prec y \Rightarrow f(x)\leq f(y) \right\}. $$

\begin{lem}\label{Halpha}
	$\mathcal{H}_\alpha$ is a Rosenthal compactum. Moreover, $\mathcal{H}_\alpha$ is a subspace of a Rosenthal compactum of the class $\mathcal{RK}_0$.
\end{lem}

\begin{proof}
		It is clear that $\mathcal{H}_\alpha$ is a  pointwise closed set of functions. If $f\not\in\mathcal{H}_\alpha$, then there exist $x\prec y$ such that $f(y)< f(x)$; taking $f(y)<r<f(x)$, the set $\{g : g(y)<r<g(x)\}$ is a neihborhood of $f$ that contains no element of $\mathcal{H}_\alpha$. We check now that each $f\in\mathcal{H}_\alpha$ is of the first Baire class. It is enough to see that $J = f^{-1}(I)$ is an $F_\sigma$ set for every open interval $I\subset [0,1]$. The set $J$ is an interval of the order $\prec$, in the sense that if $a<b<c$ and $a,c\in J$, then $b\in J$. The lexicographic order of $\{0,1\}^\alpha$ is complete, and therefore every interval can be written in the form $[a,b]$, $[a,b)$, $(a,b]$ or $(a,b)$, where $a$ and $b$ are the infimum and the supremum of $J$. We will prove that open intervals $J=(a,b)$ are $F_\sigma$, and therefore all intervals are $F_\sigma$. We have that
	\begin{eqnarray*}
		\{x : x\prec b\} &=& \bigcup_{\beta<\alpha}\{x : x_\beta =0, b_\beta = 1, x_\gamma=b_\gamma \text{ for }\gamma<\beta \}
	\end{eqnarray*}
	is a countable union of closed sets. Similarly, $\{x: a\prec x\}$ is an $F_\sigma$ as well, and therefore $(a,b)$, the intersection of the two previous sets, is $F_\sigma$ as well.
	
	This completes the proof that $\mathcal{H}_\alpha$ is a Rosenthal compactum. We have to find now a subset $\mathcal{D}$ of $[0,1]$-valued continuous functions on $\{0,1\}^\alpha$ whose pointwise closure is made of functions of the first Baire class, and contains $\mathcal{H}_\alpha$. Since $\alpha$ is countable, it can be written as the union of an increasing $\omega$-chain of finite sets: We write $\alpha = \bigcup_{n<\omega} A_n$ with $A_0\subset A_1\subset \cdots$ and each $A_i$ a finite set. For each $n<\omega$, let
	$$\mathcal{D}_n = \set{f:\{0,1\}^\alpha\To [0,1] :  x|_{A_n}\preceq y|_{A_n} \Rightarrow f(x)\leq f(y)}.$$
	Here $z|_{A_n} = (z_\beta)_{\beta\in A_n} \in \{0,1\}^{A_n}$ denotes the restriction to the cube $\{0,1\}^{A_n}$, that is viewed with the corresponding lexicographic order. The set $\mathcal{D} = \bigcup_{n<\omega}\mathcal{D}_n$ is the one we are looking for. First notice that a function $f\in\mathcal{D}_n$ satisfies that $x|_{A_n}= y|_{A_n} \Rightarrow f(x) = f(y)$, and this implies that  $f$ is continuous.  Now we prove that every function in the pointwise closure of $\mathcal{D}$ is of the first Baire class. Every such function $g$ is the limit through an ultrafilter $\mathcal{U}$ of $\mathcal{D}$. If there exists $n$ such that $\mathcal{D}_n\in\mathcal{U}$, then $g:\{0,1\}^\alpha\To [0,1]$ would keep the property that $x|_{A_n}\preceq y|_{A_n} \Rightarrow f(x)\leq f(y)$, and we noted that this implies continuity. The other possibility is that $\mathcal{D}_n\not\in\mathcal{U}$ for all $n$. We prove that in this case $g\in\mathcal{H}_\alpha$. Suppose that $x\prec y$, so that there is $\beta<\alpha$ such that $x_\beta = 0 <1 = y_\beta$, while $x_\gamma = y_\gamma$ for $\gamma<\beta$. Find $n$ such that $\beta\in A_n$. Notice that $x|_{A_m}\prec y|_{A_m}$ for all $m\geq n$. Hence $f(x)\leq f(y)$ whenever $f\in \bigcup_{m\geq n}D_m \in\mathcal{U}$. This property would be preserved in the limit, so $g(x)\leq g(y)$, and this shows that $g\in\mathcal{H}_\alpha$. The only thing that remains to be proven is that $\mathcal{H}_\alpha \subset \overline{\mathcal{D}}$. For this, fix an element $f\in\mathcal{H}_\alpha$, and a basic neighborhood of $f$ of the form $$W = \{g:\{0,1\}^\alpha\To \mathbb{R} : |g(x)-f(x)|<\varepsilon\text{ for all }x\in F \}$$
	where $\varepsilon>0$ and $F$ is a finite subset of $\{0,1\}^\alpha$. For each different $x,y\in F$ let $\gamma(x,y)$ be the least ordinal $\gamma<\alpha$ such that $x_\gamma\neq y_\gamma$. Find $n<\omega$ such that $\gamma(x,y)\in A_n$ whenever $x,y\in F$, $x\neq y$. In this way we have that, for $x,y\in F$, $x\prec y$ if and only if $x|_{A_n} \prec y|_{A_n}$. We define $g:\{0,1\}^\alpha\To[0,1]$ as
	$$g(x) = f\left(\min\{y\in F : x|_{A_n}\preceq y|_{A_n} \}\right).$$ 
	We have that $f(x)=g(x)$ when $x\in F$, and $x|_{A_n}\prec y|_{A_n}$ always implies that $g(x)\leq g(y)$. So $g\in W\cap \mathcal{D}_n$ as desired.
\end{proof}

We denote as $([0,1]^\beta,\tau_{\prec})$ the compact space $[0,1]^\beta$ when endowed with the order topology of the lexicographic product order $\prec$.

\begin{prop}\label{lexprodis}
	For every countable ordinal $\beta<\omega_1$, the compact space $([0,1]^\beta,\tau_{\prec})$ is a Rosenthal compactum that is a subspace of a Rosenthal compactum in the class $\mathcal{RK}_0$. 
\end{prop}

\begin{proof}
 First we consider the case in which $\beta=\alpha+1$ is a successor ordinal. We consider the ordinal product $$\alpha\cdot\omega = \set{ \gamma\cdot\omega+n : \gamma<\alpha, n<\omega }.$$
	Let $\pi:2^\omega\To [0,1]$ be the standard surjection $\pi(x) = \sum_{n<\omega} 2^{-n-1}x_n$, and
	 $$\varpi:\{0,1\}^{\alpha\cdot\omega} \To [0,1]^\alpha$$
	 given by $$\varpi(x) = (\pi(x_{\gamma\cdot\omega+n})_{n<\omega})_{\gamma<\alpha}.$$
	 This function is nondecreasing with respect to the lexicographic orders and continuous with respect to the induced order topologies. Now we define $\phi:[0,1]^{\alpha+1}\To \mathcal{H}_{\alpha\cdot\omega}$ by 
	 
	 $$\phi(x)(y) = \begin{cases} 
	 0 &\text{ if } \varpi(y)\prec x|_\alpha \\
	 x_\alpha &\text{ if } \varpi(y)= x|_\alpha\\ 
	 1 &\text{ if } \varpi(y)\succ x|_\alpha. 
	 \end{cases}$$
	 This is a homeomorphic embedding of $([0,1]^{\alpha+1},\tau_\prec)$ inside $\mathcal{H}_{\alpha\cdot\omega}$.
	 If $\beta$ is a limit ordinal, then we can embed $([0,1]^\beta,\tau_{\prec})$ inside $\prod_{\gamma<\beta} ([0,1]^{\gamma+1},\tau_{\prec})$. The product over $\gamma<\beta$ is now endowed with the product topology, the embedding is just given by restrictions $x\mapsto (x|_{\gamma+1})_{\gamma<\beta}$. The class $\mathcal{RK}_0$ is closed under countable products.
\end{proof}

\begin{lem}
	Let $L_1$ and $L_2$ be two complete linear orders, and $L_1\times L_2$ its lexicographic product, endowed with the order topology. If $L_1$ is uncountable and $f:L_1\times L_2\To M$ is a continuous function onto a metric space, then there exists $x\in M$ such that $f^{-1}(x)$ contains a copy of $L_2$.
\end{lem}

\begin{proof}
	Since $M$ is a compact metric space, it has a countable basis of open $F_\sigma$ sets. Taking preimages, there is a countable family $\mathcal{F}$ of open $F_\sigma$ sets in $L_1\times L_2$ such that if $f(x)\neq f(y)$ then $x$ and $y$ are separated by elements of $\mathcal{F}$. Since open intervals $(a,b)$ form a basis for the topology of $L_1\times L_2$, there is also a countable family $\{(a_n,b_n) : n<\omega\}$ of open intervals such that if $f(x)\neq f(y)$ then $x$ and $y$ are separated by these intervals. Since $L_1$ is uncountable, there exists $t\in L_1$ that is not the first coordinate of any $a_n$ or $b_n$. Then $f(t,s)=f(t,s')$ for all $s,s'\in L_2$, and this provides a copy of $L_2$ inside the preimage of a point.
\end{proof}

\begin{cor}
	Let $L_1$ and $L_2$ be two complete linear orders, and $L_1\times L_2$ its lexicographic product, endowed with the order topology. If $L_1$ is uncountable and $L_2$ is not metrizable, then $L_1\times L_2$ is not a Corson-to-one preimage of a metric space.
\end{cor}

\begin{proof}
	It follows from the previous lemma, just remember that a linearly ordered Corson compactum is metrizable \cite{CorsonLO}.
\end{proof}

The split interval and the split Cantor set, i.e. the lexicographic products $[0,1]\times\{0,1\}$ and $\{0,1\}^{\omega+1}$, are not metrizable in the order topology, therefore: 

\begin{cor}
	$(\{0,1\}^{\omega\cdot 2+1},\tau_{\prec})$ and $([0,1]^2\times\{0,1\},\tau_{\prec})$ are not Corson-to-one preimages of metric compacta.
\end{cor}

The above spaces, and the spaces in $\mathcal{RK}_0$ in which they are contained, are examples of spaces in $\mathcal{RK}\setminus\mathcal{CD}$.

\section{Continuous images} 
 
Recall that a compact space $K$ is in the class $\mathcal{CD}$ if and only if it is a pointwise compact set of functions with countably many discontinuities on a Polish space. An observation that we are going to use is that this class is closed under countable products. Indeed, if $K_n$ is a compact set of functions with countably many discontinuities on $X_n$, then $\prod_n K_n$ can be represented as the family of all functions on the disjoint topological discrete union $X$ of the $X_n$'s such that $f|_{X_n} \in K_n$ for all $n$.
 
\begin{prop}
	For every $\alpha<\omega\cdot\omega$, the compact space $\mathcal{H}_\alpha$ is a subspace of a continuous image of a separable Rosenthal compact space in the class $\mathcal{CD}$.
\end{prop} 

\begin{proof}
	We will need some extra notation to deal with the spaces described in Lemma~\ref{Halpha}. Given $\alpha<\omega_1$, $B\subset \alpha$ and $F\subset [0,1]$, let
	$$\mathcal{D}[\alpha,B,F] = \{f:\{0,1\}^\alpha\To F : x|_B \preceq y|_B \Rightarrow f(x)\leq f(y)\}.$$
	If $A_\ast = (A_n)_{n<\omega}$ is a tower of finite sets $A_0\subset A_1\subset\cdots$ such that $\bigcup A_n = \alpha$, then we define
	$$\mathcal{H}[\alpha,A_\ast,F] = \overline{\bigcup_{n<\omega}\mathcal{D}[\alpha,A_n,F]} = \bigcup_{n<\omega}\mathcal{D}[\alpha,A_n,\overline{F}] \cup \mathcal{D}[\alpha,\alpha,\overline{F}].$$
	
	The second equality was checked inside the proof of Lemma~\ref{Halpha}. Recall also that $\mathcal{H}[\alpha,A_\ast,[0,1]]$ belongs to $\mathcal{RK}_0$ and is therefore separable. It is enough to prove  that $\mathcal{H}[\alpha,A_\ast,[0,1]]$ is a continuous image of a compact space in the class $\mathcal{CD}$. Let us prove first that $\mathcal{H}[\alpha,A_\ast,F]$ is a continuous image of a space in $\mathcal{CD}$ when $F$ is finite. When $\alpha<\omega$, then all functions in  $\mathcal{H}[\alpha,A_\ast,F]$ are continuous and we are done. When $\alpha=\omega$, then the functions in $\mathcal{H}[\alpha,A_\ast,F]$ are either continuous or lexicographically monotone on $\{0,1\}^\omega$, so they have finitely many points of discontinuity and we are done again. For the rest of ordinals $\alpha<\omega\cdot\omega$, it is enough to check that if the result holds for $\beta$ and $\gamma$, then it also holds for the ordinal sum $\alpha=\beta+\gamma$. It is convenient to consider $\tilde{\gamma} =\{\beta+\xi : \xi<\gamma\}$ instead of $\gamma$. We take $G\subset [0,1]$ a finite set whose cardinality is $|G|= 2|F|+2$. Given $A_\ast$ a tower of finite sets that covers $\alpha$, we take $B_\ast = (A_n \cap \beta)_{n<\omega}$, which covers $\beta$, and $\Gamma_\ast = (A_n \cap \tilde{\gamma})_{n<\omega}$, which covers $\tilde{\gamma}$. For $f\in\mathcal{H}[\beta,B_\ast,G]$ and $g\in\mathcal{H}[\tilde{\gamma},\Gamma_\ast,G]$, we consider the function $f\boxtimes g: \{0,1\}^{\beta+\gamma}\to G\times G$ given by $f\boxtimes g (x) = (f(x|_{\beta}),g(x|_{\tilde{\gamma}}))$. We claim that
	$$(\star)\ \mathcal{H}[\alpha,A_\ast,F] \subset \set{\sigma\circ (f\boxtimes g) : f\in\mathcal{H}[\beta,B_\ast,G], g\in\mathcal{H}[\tilde{\gamma},\Gamma_\ast,G], \sigma:G\times G\To F}.$$
	For the proof of $(\star)$, it is convenient that we call $A_\omega = \alpha$. We take $h\in \mathcal{H}[\alpha,A_\ast,F]$, so that $h\in \mathcal{D}[\alpha,A_n,F]$ for some $n\leq\omega$. By montonicity, there is a partition into at most $|F|$ consecutive lexicographical intervals of $\{0,1\}^{A_n}$ so that $h$ is constant for all $x$ with $x|_{A_n}$ in one of those intervals. Let $v_1\prec\cdots\prec v_p$ in $\{0,1\}^{B_n}$ and $w_1\prec\cdots\prec w_q$ in $\{0,1\}^{\Gamma_n}$ be the respective restrictions of the endpoints of those intervals, enumerated increasingly and without repetition. We can find $f\in \mathcal{D}[\beta,B_n,G]$ whose fibers $f^{-1}(k)$, $k\in G$, are exactly the sets of the form  $(v_i,v_j)\times\{0,1\}^{\beta\setminus B_n}$ or $\{v_i\}\times\{0,1\}^{\beta\setminus B_n}$ (or empty). Here the brackets mean lexicographic open intervals. 	Similarly we can take $g\in \mathcal{D}[\tilde{\gamma},\Gamma_n,G]$  whose fibers $g^{-1}(k)$, $k\in G$, are exactly the sets of the form  $(w_i,w_j)\times\{0,1\}^{\tilde{\gamma}\setminus \Gamma_n}$ or $\{w_i\}\times\{0,1\}^{\tilde{\gamma}\setminus \Gamma_n}$ (or empty). Notice that $h$ is constant on each of the fibers of $f\boxtimes g$, so we can express $h=\sigma\circ (f\boxtimes g)$ as desired. This finishes the proof of $(\star)$.
	
	For a fixed $\sigma$, the assignment $(f,g)\mapsto \sigma\circ (f\boxtimes g)$ is continuous, and the above shows that $\mathcal{H}[\alpha,A_\ast,F]$ is covered with finitely many continuous images of the product space $\mathcal{H}[\beta,B_\ast,G]\times \mathcal{H}[\tilde{\gamma},\Gamma_\ast,G]$. So $\mathcal{H}[\alpha,A_\ast,F]$ will be a continuous image of a separable space in $\mathcal{CD}$ if both $\mathcal{H}[\beta,B_\ast,G]$ and  $\mathcal{H}[\tilde{\gamma},\Gamma_\ast,G]$ are so.  
	Finally, we deal with the case of the full $\mathcal{H}[\alpha,A_\ast,[0,1]]$. For every $n<\omega$ consider the set of dyadic rationals $F_n = \{k/2^n : k=0,1,\ldots,2^n\}$, and
	$$ L =\left\{f_\ast = (f_n)_{n<\omega}\in \prod_{n<\omega}\mathcal{H}[\alpha,A_\ast,F_n] : \forall x \ \forall n \ |f_n(x)-f_{n+1}(x)|\leq \frac{1}{2^{n+1}}  \right\}. $$ 
	
This set $L$ is a closed subspace of the countable product $\prod_{n<\omega}\mathcal{H}[\alpha,A_\ast,F_n]$, so it is a continuous image of a member of the class $\mathcal{CD}$. We define a function $\phi: L\to \mathcal{H}[\alpha,A_\ast,[0,1]]$ by $\phi((f_n)_{n<\omega}) = \lim_n f_n$. Notice that $\phi$ is continuous because for every $x\in\{0,1\}^\alpha$, the function $L\To \mathbb{R}$ given by $f_\ast\mapsto \phi(f)(x)$ is the uniform limit of the functions $f_\ast \mapsto f_n(x)$. Another remark is that $\phi$ is onto. In fact, if $g\in \mathcal{H}[\alpha,A_\ast,[0,1]]$, and we define $f_n(x) = \inf\{a\in F_n : g(x)\leq a\}$, then $f_\ast=(f_n)_{n<\omega} \in L$ and $\phi(f_\ast) = g$. This shows that $\mathcal{H}[\alpha,A_\ast,[0,1]]$ is a continuous image of a space from $\mathcal{CD}$.
\end{proof}
 
It follows from the proof of Proposition~\ref{lexprodis} that the lexicographic products $[0,1]^\alpha$, for $\alpha<\omega\cdot\omega$, are subspaces of continuous images of separable spaces in $\mathcal{CD}$. We do not know what happens for larger countable ordinals.

\end{document}